\documentclass[12pt,reqno]{amsart}

\usepackage{amssymb}
\usepackage{amsmath}
\usepackage{amscd}
\usepackage{amstext}
\usepackage{amsfonts}
\usepackage[all]{xy}

\newtheorem{thm}{Theorem}[section] 

\newtheorem{prop}[thm]{Proposition}
\newtheorem{cor}[thm]{Corollary}
\newtheorem{rem}[thm]{Remark}
\theoremstyle{definition}

\newtheorem{example}[thm]{Example}

\newcommand{\C}{\mathbb{C}}
\newcommand{\R}{\mathbb{R}}

\newcommand{\Q}{\mathbb{Q}}
\newcommand{\Z}{\mathbb{Z}}
\newcommand{\g}{\frak{g}}

\numberwithin{equation}{section}

\begin{document} 

\title[Line bundles and Hodge symmetry on Oeljeklaus-Toma manifolds]{Cohomology of  holomorphic line bundles and Hodge symmetry on Oeljeklaus-Toma manifolds}

\author[H. Kasuya]{Hisashi Kasuya}

\address{Department of Mathematics, Graduate School of Science, Osaka University, Osaka,
Japan}

\email{kasuya@math.sci.osaka-u.ac.jp}



\begin{abstract}
We prove the Hodge symmetry type result on the Dolbeault cohomology of Oeljeklaus-Toma manifolds with values in the direct sum of holomorphic line bundles.
Consequently, we show the vanishing and non-vanishing of Dolbeault cohomology of Oeljeklaus-Toma manifolds with values in  holomorphic line bundles.

\end{abstract}

\maketitle
\section{Introduction}

For positive integers $s,t$, let $K$ be a finite extension field of $\Q$ of degree $s+2t$
 admitting  embeddings $\sigma_{1},\dots \sigma_{s},\sigma_{s+1},\dots, \sigma_{s+2t}$ into $\C$ such that $\sigma_{1},\dots ,\sigma_{s}$ are real embeddings and $\sigma_{s+1},\dots, \sigma_{s+2t}$ are complex ones satisfying $\sigma_{s+i}=\bar \sigma_{s+i+t}$ for $1\le i\le t$. 
Denote by ${\mathcal O}_{K}$  the ring of algebraic integers of $K$.
In \cite{OT}, for a free subgroup $U$ of rank $s$ in  the group of units in ${\mathcal O}_{K}$ satisfying certain conditions related to the embeddings $\sigma_{1},\dots \sigma_{s},\sigma_{s+1},\dots, \sigma_{s+2t}$, 
 Oeljeklaus and  Toma construct a $(s+t)$-dimensional complex manifold $X(K,U)$ whose fundamental group is the semi-direct product $U\ltimes {\mathcal O}_{K}$ of finitely generated free abelian groups $U$ and  ${\mathcal O}_{K}$.
 We call this complex manifold an  Oeljeklaus-Toma (OT) manifold.
 For any OT-manifold $X(K,U)$, the Hodge symmetry 
 \[\dim H^{p,q}(X(K,U))=\dim H^{q,p}(X(K,U))\]
  on the Dolbeault cohomology does not hold (see \cite[Proposition 2.5]{OT}).
 Hence every OT-manifold $X(K,U)$ is a non-K\"ahler complex manifold.
Meanwhile, in \cite{OTH}, Otiman and Toma show that the Hodge decomposition 
\[\dim H^{r}(X(K,U),\C)=\sum_{p+q=r} \dim H^{p,q}(X(K,U))\]
 holds for any OT-manifold $X(K,U)$.

 Consider the set ${\rm Hom}(U,\C^{\ast})$ of group homomorphisms from $U$ to $\C^{\ast}$.
Corresponding representations of the fundamental group of $X(K,U)$ to flat bundles over $X(K,U)$, this set is identified with a set ${\mathcal A}(U)$ of isomorphism classes of flat complex line bundles over  $X(K,U)$.
 For $E\in {\mathcal A}(U)$,  we consider the de Rham cohomology $H^{\ast}(X(K,U),  E)$ with values in $E$.
 We have the following result.
 \begin{prop}\label{Pco}
 For any integer $r$, we have
\[\dim \bigwedge ^{r}\C^{2s+2t} = \sum_{E\in {\mathcal A}(U)}\dim H^{r}(X(K,U),E).
\] 
 \end{prop}
 As explained in \cite{KV}, every OT-manifold is diffeomorphic  to a solvmanifold.
This proposition is a consequence of the cohomology computation of solvmanifolds in \cite{KR}. 
 
Regarding each $E\in {\mathcal A}(U)$ as a holomorphic line bundle over $X(K,U)$, we consider the Dolbeault cohomology $H^{p,q}(X(K,U),E)=H^{q}(X(K,U),\Omega^{p}(E))$.
In this paper, we prove the following Hodge symmetry type result.
\begin{thm}\label{Mathe}
For any integers $p,q$,  we  have 
   \[\dim \bigwedge ^{p}\C^{s+t}\otimes \bigwedge^{q}\C^{s+t} = \sum_{E\in {\mathcal A}(U)}\dim H^{p,q}(X(K,U),E).\] 
\end{thm}

More precisely, we can give explicit harmonic representatives of $\bigoplus_{{\mathcal A}(U)} H^{p,q}(X(K,U),E)$.
In particular, we know the vanishing or non-vanishing of the cohomology $ H^{p,q}(X(K,U),E)$ for any $E\in {\mathcal A}(U)$.
For a positive integer $s$, we denote  $[s]=\{1,2,\dots, s\}$ and we call a subset  in $[s]$ with the natural order a multi-index.

\begin{thm}\label{vanan}
Let $E$ be a flat complex line bundle over an OT-manifold $X(K,U)$ corresponding to $\rho\in {\rm Hom}(U,\C^{\ast})$.
Then
 \[
H^{p,q}(X(K,U),E)\not=0\]
 if and only if   for some multi-indices  $I\subset [s], K,L\subset [t]$   with $\vert I\vert+\vert K\vert=p$ and $\vert L \vert\le q$,  we have
 \[\rho(u)=\prod_{i\in I}\sigma_{i}(u)\prod_{k\in K}\sigma_{s+k}(u)\prod_{l\in L}\sigma_{s+t+l}(u)\]
  for any $u\in U$.  
   If 
  \[\rho(u)=\prod_{i\in I}\sigma_{i}(u)\prod_{k\in K}\sigma_{s+k}(u)\prod_{l\in L}\sigma_{s+t+l}(u)\]
   for any $u\in U$,
  then  we have 
   \[\dim H^{p,q}(X(K,U),E)\ge \left( \begin{array}{cc}
s  \\
q-\vert L \vert
\end{array}\right)
\]
where $ \left( \begin{array}{cc}
n  \\
k
\end{array}\right)$ means the number of $k$-combinations.

\end{thm}
Obviously we have the following consequence.
\begin{cor}
Let $E$ be a flat complex line bundle over an OT-manifold $X(K,U)$ corresponding to $\rho\in {\rm Hom}(U,\C^{\ast})$.
Then, $E$ admits a non-zero holomorphic section if and only if $\rho$ is trivial.
\end{cor}

As a corollary of Theorem \ref{vanan}, we also obtain 
\cite[Theorem 3.1]{APV} (see Remark \ref{APV}).

\begin{rem}
Our main results are obtained by using the solvmanifold presentations of  OT manifolds.
We give statements in terms of solvmanifolds (Proposition \ref{III}, Theorem \ref{Dol} and  Theorem \ref{Van}) implying Proposition \ref{Pco}, Theorem  \ref{Mathe} and Theorem \ref{vanan}.

We notice that  Theorem \ref{Mathe} is deeper than Proposition \ref{Pco}.
In \cite{KV}, we also compute the cohomology of complex parallelizable solvmanifolds with values in holomorphic vector bundles.
OT-manifolds do not admit non-zero holomorphic vector field (\cite[Proposition 2.5]{OT}) and they are far from complex parallelizable solvmanifolds.
To prove Theorem \ref{Mathe}, we use the result in \cite{OTH} given by the analysis on Cousin groups.
\end{rem}

\section{Cohomology of solvmanifolds}
Let $G$ be a simply connected  real solvable Lie group with a lattice $\Gamma$ and $\g$ the Lie algebra of $G$.
Denote by $A^{\ast}(\Gamma \backslash G)$ the de Rham complex of the solvmanifold  $\Gamma\backslash G$.
We identify the de Rham complex $A^{\ast}(\Gamma\backslash G)$ of $\Gamma\backslash G$ with the subcomplex of  the de Rham complex $A^{\ast}( G)$ of  $G$ consisting of the left-$\Gamma$-invariant differential forms.
Consider the cochain complex $\bigwedge \g^{\ast}$ of the Lie algebra $\g$.
We identify $\bigwedge \g^{\ast}$ with the subcomplex of  the de Rham complex $A^{\ast}( G)$ of  $G$ consisting of the  left-$G$-invariant differential forms.
Hence, we consider $\bigwedge \g^{\ast}$ as a subcomplex of  $A^{\ast}(\Gamma\backslash G)$.

Let $N$ be the nilradical (i.e. maximal connected nilpotent normal subgroup)  of $G$.
Denote ${\mathcal A}_{(G,N)}=\{\alpha\in {\rm Hom}(G,\C^{\ast}) \vert \alpha_{\vert_{N}}=1\}$.
For $\alpha\in {\mathcal A}_{(G,N)}$, we consider  the $\Gamma$-action on $G\times \C$ so that for $\gamma\in \Gamma$  and $(g,c)\in G\times \C$, $\gamma\cdot (g,c)=(\gamma g, \alpha(\gamma)c)$.
Define the flat complex line bundle $E_{\alpha}=\Gamma\backslash(G\times \C) $ over the solvmanifold $\Gamma \backslash G$.
We have the global section $v_{\alpha}$ induced by the section $(g,\alpha(g))$ of $G\times \C$.
Denote by $A^{\ast}(\Gamma \backslash G, E_{\alpha})$ the de Rham complex with values in the line bundle $E_{\alpha}$.
Since $v_{\alpha}$ trivializes $E_{\alpha}$, we have $A^{\ast}(\Gamma \backslash G, E_{\alpha})=A^{\ast}(\Gamma \backslash G)\otimes \langle v_{\alpha}\rangle$ and $d v_{\alpha}=\alpha^{-1}d\alpha v_{\alpha}$.
We have the subcomplex $\bigwedge \g^{\ast}_{\C}\otimes  \langle v_{\alpha}\rangle\subset A^{\ast}(\Gamma \backslash G, E_{\alpha})=A^{\ast}(\Gamma \backslash G)\otimes \langle v_{\alpha}\rangle$.
The cochain complex  $\bigwedge \g^{\ast}_{\C}\otimes  \langle v_{\alpha}\rangle$ is the cochain complex of the Lie algebra $\g$ associated with the representation $\alpha\in {\rm Hom}(G,\C^{\ast})$.

We define  ${\mathcal A}_{(G,N)}(\Gamma)$ by  the set $\{E_{\alpha}\}$ of all the isomorphism classes of flat line bundles $E_{\alpha}$ associated with $\alpha\in {\mathcal A}_{(G,N)}$.
This set  is identified with the set  $\{\alpha_{\vert\Gamma}\in {\rm Hom}(\Gamma,\C^{\ast})\vert \alpha\in {\mathcal A}_{(G,N)}\}$.
Consider the direct sum
\[\bigoplus_{E_{\alpha}\in {\mathcal A}_{(G,N)}(\Gamma)}A^{\ast}(\Gamma \backslash G, E_{\alpha}).
\]
By the natural isomorphisms $E_{\alpha}\otimes E_{\beta}\cong E_{\alpha\beta}$ for $\alpha,\beta\in {\mathcal A}_{(G,N)}$,
this direct sum is a differential graded algebra. 
By the above argument, we have the inclusion 
\[\bigoplus_{\alpha\in {\mathcal A}_{(G,N)}}\bigwedge \g^{\ast}\otimes  \langle v_{\alpha}\rangle
\subset \bigoplus_{E_{\alpha}\in {\mathcal A}_{(G,N)}(\Gamma)}A^{\ast}(\Gamma \backslash G, E_{\alpha}).
\]

We have a simply connected nilpotent subgroup $C\subset G$ such that $G=C\cdot N$ (see \cite[Proposition 3.3]{dek}).
Since $C$ is nilpotent, the map 
\[\Phi:C\ni c \mapsto  ({\rm Ad}_{c})_{s}\otimes \alpha(c) \in {\rm Aut}\left(\bigwedge {\frak g}^{\ast}_{\C}\otimes \langle v_{\alpha}\rangle\right)\]
is a homomorphism where $({\rm Ad}_{c})_{s}$ is the semi-simple part of the Jordan decomposition of the adjoint operator.
We denote by 
\[\left(\bigwedge {\frak g}^{\ast}_{\C}\otimes \langle v_{\alpha}\rangle\right)^{\Phi(C)}
\]
the subcomplex of $\bigwedge {\frak g}^{\ast}_{\C}\otimes \langle v_{\alpha}\rangle$ consisting of the $\Phi(C)$-invariant elements.
\cite[Theorem 1.4, Lemma 5.2]{KR} says that the inclusion
\[\bigoplus_{\alpha\in {\mathcal A}_{(G,N)}}\left(\bigwedge {\frak g}^{\ast}_{\C}\otimes \langle v_{\alpha}\rangle\right)^{\Phi(C)}\subset  \bigoplus_{E_{\alpha}\in {\mathcal A}_{(G,N)}(\Gamma)}A^{\ast}(\Gamma \backslash G, E_{\alpha})
\]
induces a cohomology isomorphism.

We have a basis $X_{1},\dots,X_{n}$ of $\g_{\C}$ such that $({\rm Ad}_{c})_{s}={\rm diag} (\alpha_{1}(c),\dots,\alpha_{n}(c))$ for all $c\in C$.
Let $x_{1},\dots, x_{n}$ be the basis of $\g_{\C}^{\ast}$ which is dual to $X_{1},\dots ,X_{n}$.
By $G=C\cdot N$, we have $G/N=C/C\cap N$ and hence we have ${\mathcal A}_{(G,N)}={\mathcal A}_{C,C\cap N}=\{\alpha\in {\rm Hom}(C,\C^{\ast}) \vert \alpha_{\vert_{C\cap N}}=1\}$.
For a multi-index $I=\{i_{1},\dots ,i_{p}\}\subset [n]$ we write $x_{I}=x_{i_{1}}\wedge\dots \wedge x_{i_{p}}$,  and $\alpha_{I}=\alpha_{i_{1}}\cdots \alpha_{i_{p}}$.
We consider the basis 
\[\{ x_{I} \otimes v_{\alpha}\}_{I\subset [n]}
\]
of $\bigwedge {\frak g}^{\ast}_{\C}\otimes \langle v_{\alpha}\rangle$.
Since the action
\[\Phi:C\to{\rm Aut}\left(\bigwedge {\frak g}^{\ast}_{\C}\otimes \langle v_{\alpha}\rangle\right)\]
is  the semi-simple part of $({\rm Ad}\otimes \alpha)_{\vert_C}$, we have
\[\Phi(a)(x_{I} \otimes v_{\alpha})
=\alpha^{-1}_{I} \alpha x_{I}\otimes v_{\alpha}.
\]

Hence we have
\begin{multline*}
\bigoplus_{\alpha\in {\mathcal A}_{(G,N)}}\left(\bigwedge {\frak g}^{\ast}_{\C}\otimes \langle v_{\alpha}\rangle\right)^{\Phi(C)}
=\langle x_{I}\otimes v_{\alpha_{I}} \rangle_{I\subset [n]}
=\bigwedge \langle x_{1}\otimes v_{\alpha_{1}},\dots, x_{n} \otimes v_{\alpha_{n}} \rangle .
\end{multline*}
It is known that the differential graded algebra 
\[\bigwedge \langle x_{1}\otimes v_{\alpha_{1}},\dots, x_{n} \otimes v_{\alpha_{n}} \rangle
\]
is identified with  the cochain complex of certain nilpotent Lie algebra  determined by the solvable Lie algebra $\g$ (see \cite[Remark 4]{KR} and \cite{KM}).

\section{de Rham and Dolbeault cohomology of certain solvmanifolds}

Let $s,t$ be positive integers.
We consider the semi-direct product $G=\R^{s}\ltimes_{\phi} (\R^{s}\oplus \C^{t})$ of real abelian Lie groups $\R^{s}$ and $\R^{s}\oplus \C^{t}$ given by the homomorphism $\phi:\R^{s}\to {\rm Aut}(\R^{s}\oplus \C^{t})$ so that 
\[\phi(x)(y, z)=(e^{x_{1}}y_{1},\dots ,e^{x_{s}}y_{s}, e^{\psi_{1}(x)}z_{1},\dots,  e^{\psi_{t}(x)}z_{t})
\]
for $x=(x_{1},\dots ,x_{s})\in \R^{s}$, $(y, z)=(y_{1},\dots ,y_{s}, z_{1},\dots,  z_{t})\in \R^{s}\oplus \C^{t}$ and some non-zero linear 
functions $\psi_{1},\dots, \psi_{t}: \R^{s}\to  \C$.
$G$ is a simply connected solvable Lie group.
Suppose we have lattices $\Lambda\subset \R^{s}$ and $\Delta \subset \R^{s}\oplus \C^{t}$ so that for every $\lambda\in \Lambda $ the automorphism  $\phi(\lambda)$ on $\R^{s}\oplus \C^{t}$ preserves $\Delta$.
Then the subgroup $\Gamma=\Lambda\ltimes_{\phi} \Delta\subset G$ is a cocompact discrete subgroup of $G$.
Consider the solvmanifold $\Gamma\backslash G$.

We have a basis 
\[ dx_{1},\dots, dx_{s},e^{-x_{1}}dy_{1},\dots, e^{-x_{s}}dy_{s}, e^{-\psi_{1}(x)}dz_{1},\dots,  e^{-\psi_{t}(x)}dz_{t}, e^{-\bar\psi_{1}(x)}d\bar{z}_{1},\dots , e^{-\bar\psi_{t}(x)}d\bar{z}_{t}\]
 of $\g^{\ast}_{\C}=\g^{\ast}\otimes \C$.

On  $G=\R^{n}\ltimes_{\phi} (\R^{s}\oplus \C^{t})$, the nilradical $N$ is $\R^{s}\oplus \C^{t}$ and we  take the subgroup $C=\R^{s}\subset G$ so that $C\cdot N=G$.
We apply the last section to the solvmanifold $\Gamma\backslash G$.
By the last section, defining
\[V= \left\langle\begin{array}{cccc} dx_{1},\dots, dx_{s}, \\
 e^{-x_{1}}dy_{1}\otimes v_{e^{x_{1}}},\dots,e^{-x_{s}}dy_{s}\otimes v_{e^{x_{s}}}, \\
   e^{-\psi_{1}(x)}dz_{1}\otimes v_{e^{\psi_{1}(x)}},\dots ,  e^{-\psi_{t}(x)}dz_{t}\otimes v_{e^{\psi_{t}(x)}} ,\\
    e^{-\bar\psi_{1}(x)}d\bar{z}_{1}\otimes v_{e^{\bar\psi_{1}(x)}}, \dots,  e^{-\bar\psi_{t}(x)}d\bar{z}_{t}\otimes v_{e^{\bar\psi_{t}(x)}}  \end{array} \right\rangle
\]
where we regard $ dx_{1},\dots, dx_{s}$ as $1$-forms with values in the trivial line bundle,
 we have the inclusion
\[\bigwedge V \subset \bigoplus_{E_{\alpha}\in {\mathcal A}_{(G,N)}(\Gamma)}A^{\ast}(\Gamma \backslash G, E_{\alpha})
\]
which induces a cohomology isomorphism.
We notice that we can identify ${\mathcal A}_{(G,N)}$ with the set ${\rm Hom}(\R^{s}, \C^{\ast})$ of Lie group homomorphisms from $\R^{s}$ to $\C^{\ast}$ and the set ${\mathcal A}_{(G,N)}(\Gamma)$ is equal to  the set  of isomorphism classes of flat complex line bundles over  $\Gamma \backslash G$ given by homomorphisms in ${\rm Hom}(\Lambda,\C^{\ast})$.
We can easily check that the differential on $V$ is $0$.
 Hence we have the following:
\begin{prop}\label{III}
   We have an isomorphism
 \[\bigwedge V \cong  \bigoplus_{E_{\alpha}\in {\mathcal A}_{(G,N)} }H^{\ast}(\Gamma \backslash G, E_{\alpha}).\]   
\end{prop}

Regarding $1$-forms 
\[\alpha_{1}= dx_{1}+\sqrt{-1}e^{-x_{1}}dy_{1},\dots, \alpha_{s}=dx_{s}+\sqrt{-1}e^{-x_{s}}dy_{s}, \beta_{1}= e^{-\psi_{1}(x)}dz_{1},\dots ,\beta_{t}= e^{-\psi_{t}(x)}dz_{t}\]
as $(1,0)$-forms on $\Gamma\backslash G$, 
we have a left-$G$-invariant almost complex structure $J$ on $\Gamma\backslash G$.
We can easily check that $J$ is integrable.
We consider the Dolbeault complex $(A^{\ast,\ast} (\Gamma\backslash G), \bar\partial)$ of the complex manifold $(\Gamma\backslash G, J)$.
We have 
\[\bar\partial\alpha_{i}=-\frac{1}{2} \bar\alpha_{i}\wedge \alpha_{i},\qquad \bar\partial\bar{\alpha}_{i}=0,\qquad \bar\partial\beta_{i}=-\frac{1}{2}\psi_{i}(\bar\alpha)\wedge \beta_{i}\,\, {\rm and}\, \, \bar\partial\bar{\beta}_{i}=-\frac{1}{2}\bar{\psi}_{i}(\bar\alpha)\wedge \bar{\beta}_{i}
\]
where $\psi_{i}(\bar\alpha)$ and $\bar{\psi}_{i}(\bar\alpha)$ are $(0,1)$-forms associated with linear functions $\psi_{i}(x)$ and $\bar{\psi}_{i}(x)$ by putting $x=\bar\alpha=(\bar\alpha_{1},\dots, \bar\alpha_{s})$.

\begin{rem}\label{tan}
We consider the holomorphic tangent bundle $\Theta$ and holomorphic cotangent bundle $\Omega^{1}$ of $(\Gamma\backslash G, J)$.
Then,  $\alpha_{1},\dots,\alpha_{s},\beta_{1},\dots,\beta_{t}$ is a global ${\mathcal C}^{\infty}$-frame of $\Omega^{1}$.
Hence we have an isomorphism 
\[\Omega^{1}\cong E_{e^{-x_{1}}}\oplus \dots \oplus E_{e^{-x_{s}}}\oplus E_{e^{-\psi_{1}(x)}}\oplus \dots \oplus E_{e^{-\psi_{t}(x)}}\]
 of holomorphic vector bundles.
By this, we also have
\[\Theta \cong E_{e^{x_{1}}}\oplus \dots \oplus E_{e^{x_{s}}}\oplus E_{e^{\psi_{1}(x)}}\oplus \dots \oplus E_{e^{\psi_{t}(x)}}.\]
\end{rem}

For any  $e^{\Psi(x)}\in {\mathcal A}_{(G,N)}={\rm Hom}(\R^{s}, \C^{\ast})$ associated with a complex valued linear function $\Psi(x)$ on $\R^{s}$, we regard the flat line bundle $E_{e^{\Psi(x)}}$
as a holomorphic line bundle over the complex manifold $(\Gamma\backslash G, J)$.
We have $\bar\partial v_{e^{\Psi(x)}}=\frac{1}{2}\Psi(\bar\alpha) \otimes v_{e^{\Psi(x)}}$.
Define 
\begin{equation}\label{WW}
W_{1}= \left\langle\begin{array}{cccc} \alpha_{1}\otimes v_{e^{x_{1}}},\dots,  \alpha_{s}\otimes v_{e^{x_{s}}},\\
   \beta_{1}\otimes v_{e^{\psi_{1}(x)}},\dots ,  \beta_{t}\otimes v_{e^{\psi_{t}(x)}}  \end{array} \right\rangle ,
   W_{2}= \left\langle\begin{array}{cccc}
 \bar\alpha_{1},\dots,\bar\alpha_{s}, \\
   \bar\beta_{1}\otimes v_{e^{\bar\psi_{1}(x)}}, \dots,  \bar\beta_{t}\otimes v_{e^{\bar\psi_{t}(x)}}  \end{array} \right\rangle
\end{equation}
where we regard $ \bar\alpha_{1},\dots,\bar\alpha_{s}$ as $1$-forms with values in the trivial line bundle.
We consider the subspace 
\[\bigwedge^{p} W_{1}\otimes \bigwedge^{q} W_{2} \subset \bigoplus_{E_{\alpha}\in {\mathcal A}_{(G,N)}(\Gamma)}A^{p,q}(\Gamma \backslash G, E_{\alpha}).
\]

Define the left-$G$-invariant  Hermitian metric 
\[h_{G}=\alpha_{1}\cdot \bar\alpha_{1}+\dots +\alpha_{s}\cdot \bar\alpha_{s}+\beta_{1}\cdot \bar\beta_{1}+\dots +\beta_{t}\cdot \bar\beta_{t}.
\]
Define the Hermitian metric $h_{\alpha}$ on each $E_{\alpha}\in $ so that $h_{\alpha}(v_{\alpha},v_{\alpha})=1$.
We notice that for $\alpha, \alpha^{\prime}\in {\mathcal A}_{(G,N)}$, if $E_{\alpha}=E_{\alpha^{\prime}}$, then $h_{\alpha}=h_{\alpha^{\prime}}$ since $E_{\alpha}=E_{\alpha^{\prime}}$ if and only if  $\alpha_{\vert \Gamma}=\alpha^{\prime}_{\vert\Gamma}$ and hence $\alpha^{-1}\alpha^{\prime}$ is unitary.
We consider the  Hodge star operator  $\bar\ast:A^{p,q}(\Gamma \backslash G, E_{\alpha})\to  A^{s+t-p,s+t-q}(\Gamma \backslash G, E_{\alpha}^{\ast})$ associated with this metric.
Then we have
\[\bar\ast \left( \alpha_{I}\wedge \bar\alpha_{J}\wedge \beta_{K}\wedge \bar{\beta}_{L} \otimes v_{e^{\Psi_{IK\bar{L}}(x)}}\right)=\pm\alpha_{\check{I}}\wedge \bar\alpha_{\check{J}}\wedge \beta_{\check{K}}\wedge \bar{\beta}_{\check{L}}\otimes v_{e^{-\Psi_{IK\bar{L}}(x)}}
\]
where for some multi-indices $I,J\subset [s], K,L\subset [t]$ we write 
\[\Psi_{IK\bar{L}}(x)=\sum_{j\in I}x_{j}+\sum_{k\in K}\psi_{k}(x)+\sum_{l\in L}\bar{\psi}_{l}(x),
\]
$\check{I}=[s]-I$, $\check{J}=[s]-J$, $\check{K}=[t]-K$ and  $\check{L}=[t]-L$.
Since $G$ admits a lattice $\Gamma$, $G$ is unimodular (see \cite[Remark 1.9]{R}).
This implies 
\[\exp({\Psi}_{[s][t]\overline{[t]}}(x))=\exp\left(\sum_{i\in [s]}x_{i}+\sum_{k\in [t]}\psi_{k}(x)+\sum_{l\in [t]}\bar{\psi}_{l}(x)\right)=1.\]
Thus 
\[\exp(\Psi_{\check{I}\check{K}\overline{\check{L}}}(x))= \exp(-\Psi_{IK\bar{L}}(x)).
\]
By this, we can say that  the  Hodge star operator  $\bar\ast$ preserves the space
$\bigwedge W_{1}\otimes \bigwedge W_{2}$ (compare \cite[Lemma 2.3]{KRO}).

We can easily check that the Dolbeault operator on 
$W_{1} $ and $W_{2}$ is $0$.
Hence, $\bigwedge W_{1}\otimes \bigwedge W_{2}$ consists of harmonic forms associated with the Dolbeault operator.
This implies the following result.
\begin{prop}\label{inj}
We have an injection
\[ \bigwedge^{p} W_{1}\otimes \bigwedge^{q} W_{2} \hookrightarrow \bigoplus_{E_{\alpha}\in {\mathcal A}_{(G,N)}(\Gamma)}H^{p,q}(\Gamma \backslash G, E_{\alpha})
   \]
   hence we have 
   \[\dim \bigwedge ^{p}\C^{s+t}\otimes \bigwedge^{q}\C^{s+t} \le \sum_{E_{\alpha}\in {\mathcal A}_{(G,N)}(\Gamma)}\dim H^{p,q}(\Gamma \backslash G, E_{\alpha}). \]
\end{prop}

\section{Oeljeklaus-Toma manifolds}
For positive integers $s,t$,  let $K$ be a finite extension field of $\Q$ of degree $s+2t$ 
 admitting  embeddings $\sigma_{1},\dots \sigma_{s},\sigma_{s+1},\dots, \sigma_{s+2t}$ into $\C$ such that $\sigma_{1},\dots ,\sigma_{s}$ are real embeddings and $\sigma_{s+1},\dots, \sigma_{s+2t}$ are complex ones satisfying $\sigma_{s+i}=\bar \sigma_{s+i+t}$ for $1\le i\le t$. 
Let ${\mathcal O}_{K}$ be the ring of algebraic integers of $K$, ${\mathcal O}_{K}^{\ast}$ the group of units in ${\mathcal O}_{K}$ and 
\[{\mathcal O}_{K}^{\ast\, +}=\{a\in {\mathcal O}_{K}^{\ast}: \sigma_{i}(a)>0 \,\, {\rm for \,\,  all}\,\, 1\le i\le s\}.
\]  
Define $\sigma :{\mathcal O}_{K}\to \R^{s}\times \C^{t}$ by
\[\sigma(a)=(\sigma_{1}(a),\dots ,\sigma_{s}(a),\sigma_{s+1}(a),\dots ,\sigma_{s+t}(a))
\]
for $a\in {\mathcal O}_{K}$.
Define $l:{\mathcal O}_{K}^{\ast\, +}\to \R^{s+t}$ by 
\[l(a)=(\log \sigma_{1}(a),\dots ,\log  \sigma_{s}(a), 2\log \vert \sigma_{s+1}(a)\vert,\dots ,2\log \vert \sigma_{s+t}(a)\vert)
\]
for $a\in {\mathcal O}_{K}^{\ast\, +}$.
Then by Dirichlet's units theorem, $l({\mathcal O}_{K}^{\ast\, +})$ is a lattice in the vector space $L=\{x\in \R^{s+t}\vert \sum_{i=1}^{s+t} x_{i}=0\}$.
Consider the projection $p:L\to \R^{s}$ given by the first $s$ coordinate functions.
Then we have a  subgroup $U$ with the rank $s$ of ${\mathcal O}_{K}^{\ast\, +}$ such that $p(l(U))$ is a lattice in $\R^{s}$.
Write $l(U)=\Z v_{1}\oplus\dots\oplus \Z v_{s}$ for generators $v_{1},\dots v_{s}$ of $l(U)$.
For the standard basis $e_{1},\dots ,e_{s+t}$ of $\R^{s+t}$, we have a regular real $s\times s$-matrix $(a_{ij})$ and $s\times t$ real constants  $b_{jk}$ such that
\[v_{i}=\sum_{j=1}^{s} a_{ij}(e_{j}+\sum_{k=1}^{t}b_{jk}e_{s+k})
\]
for any $1\le i\le s$.
Consider the complex upper half plane $H=\{z\in \C: {\rm Im} z>0\}=\R\times \R_{>0}$.
We have the left action of $U\ltimes{\mathcal O}_{K}$ on $H^{s}\times \C^{t}$
such that 
\begin{multline*}
(a,b)\cdot (x_{1}+\sqrt{-1}y_{1},\dots ,x_{s}+\sqrt{-1}y_{s}, z_{1},\dots ,z_{t})\\
=(\sigma_{1}(a)x_{1}+\sigma_{1}(b)+\sqrt{-1} \sigma_{1}(a)y_{1}, \dots ,\sigma_{s}(a)x_{s}+\sigma_{s}(b)+\sqrt{-1} \sigma_{s}(a)y_{s},\\
 \sigma_{s+1}(a)z_{1}+\sigma_{s+1}(b),\dots ,\sigma_{s+t}(a)z_{t}+\sigma_{s+t}(b)).
\end{multline*}
In \cite{OT} it is proved that the quotient $X(U, K)=U\ltimes{\mathcal O}_{K} \backslash H^{s}\times \C^{t}$ is compact.
Actually we have the real fiber bundle $X(U, K)\to U\backslash (\R_{>0})^{s}$  with the fiber $\sigma({\mathcal O}_{K}) \backslash  (\R^{s}\times \C^{t})$ and both the base $U\backslash (\R_{>0})^{s}$ and the fiber $\sigma({\mathcal O}_{K}) \backslash  (\R^{s}\times \C^{t})$  are real tori.
We call this complex manifold an  Oeljeklaus-Toma (OT) manifold.

As in \cite{KV}, we present OT-manifolds as solvmanifolds considered  in the last section.
For $a \in U$ and $(x_{1},\dots ,x_{s})=p(l(a))\in p(l(U))$, since $l(U)$ is generated by the basis $v_{1},\dots,v_{s}$ as above, $l(a)$ is a linear combination of $e_{1}+\sum_{k=1}^{t}b_{1k}e_{s+k},\dots ,e_{s}+\sum_{k=1}^{t}b_{sk}e_{s+k}$ and hence we have 
\[l(a)=\sum_{i=1}^{s}x_{i}(e_{i}+\sum_{k=1}^{t}b_{ik}e_{s+k})=(x_{1},\dots ,x_{s}, \sum_{i=1}^{s}b_{i1}x_{i},\dots,\sum_{i=1}^{s}b_{it}x_{i}).
\]
By $2\log \vert \sigma_{s+k}(a)\vert=\sum_{i=1}^{s}b_{ik}x_{i}$, we can write 
\[\sigma_{s+k}(a)=e^{\frac{1}{2}\sum_{i=1}^{s}b_{ik}x_{i}+\sqrt{-1}\sum_{i=1}^{s}c_{ik}x_{i}}\]
for some $c_{ik}\in \R$.
We consider the Lie group $G=\R^{s}\ltimes_{\phi} (\R^{s}\times \C^{t})$ with
\[
\phi(x_{1},\dots ,x_{s})\\
={\rm diag}(e^{x_{1}},\dots ,e^{x_{s}},e^{\psi_{1}(x)},\dots ,e^{\psi_{t}(x)})\]
 where $\psi_{k}=\frac{1}{2}\sum_{i=1}^{s}b_{ik}x_{i}+\sqrt{-1}\sum_{i=1}^{s}c_{ik}x_{i}$.
Then for $(x_{1},\dots ,x_{s})\in p(l(U))$, we have 
\[\phi(x_{1},\dots ,x_{s})(\sigma ({\mathcal O}_{K}))\subset \sigma({\mathcal O}_{K}).\]
Write $ p(l(U))=\Lambda$ and $ \sigma({\mathcal O}_{K})=\Delta$.
Then, via the diffeomorphism
\begin{multline*}H^{s}\times \C^{t}\ni (y_{1}+\sqrt{-1}w_{1},\dots, y_{s}+\sqrt{-1}w_{s}, z_{1},\dots, z_{t})\\
\mapsto (\log(w_{1}),\dots,\log(w_{s}),-y_{1},\dots, -y_{s},z_{1},\dots, z_{t})\in \R^{s}\ltimes_{\phi} (\R^{s}\times \C^{t}),
\end{multline*}
 the OT-manifold $X(U, K)=U\ltimes{\mathcal O}_{K} \backslash H^{s}\times \C^{t}$ is identified with a complex  solvmanifold  $(\Gamma\backslash G, J)$ of the form  as in the last section.

\begin{thm}\label{Dol}
Define $W_{1}$ and $W_{2}$ as (\ref{WW}).
An isomorphism
\[\bigwedge^{p} W_{1}\otimes \bigwedge^{q} W_{2}\cong \bigoplus_{E_{\alpha}\in{\mathcal A}_{(G,N)}(\Gamma)}H^{p,q}(\Gamma \backslash G, E_{\alpha})
   \]
holds.
Hence we have 
   \[\dim \bigwedge ^{p}\C^{s+t}\otimes \bigwedge^{q}\C^{s+t} = \sum_{E_{\alpha}\in{\mathcal A}_{(G,N)}(\Gamma)}\dim H^{p,q}(\Gamma \backslash G, E_{\alpha}).\] 
\end{thm}

\begin{proof}
We regard $H^{p,q}(\Gamma \backslash G, E_{\alpha})$ as the sheaf cohomology $H^{q}(\Gamma \backslash G, \Omega^{p}(E_{\alpha}))$.
We consider the real fiber bundle $ \pi:\Gamma \backslash G\to \Lambda\backslash\R^{s}$ over the real torus $ \Lambda\backslash\R^{s}$  with the real torus fiber $\Delta\backslash  (\R^{s}\times \C^{t})$.
This fiber bundle is identified with the fiber bundle $X(U, K)\to U\backslash (\R_{>0})^{s}$  with the fiber $\sigma({\mathcal O}_{K}) \backslash  (\R^{s}\times \C^{t})$.
Consider the Leray spectral sequence $E^{\ast,\ast}_{\ast}(\Omega^{p}(E_{\alpha}))$ associated with the map $ \pi:\Gamma \backslash G\to \Lambda\backslash\R^{s}$ and the sheaf $\Omega^{p}(E_{\alpha})$.
Then $E^{a,b}_{2}(\Omega^{p}(E_{\alpha}))=H^{a}(\Lambda\backslash\R^{s}, R^{b}\pi_{\ast} \Omega^{p}(E_{\alpha}))$ and $E^{a,b}_{r}(\Omega^{p}(E_{\alpha}))$ converges to $H^{a+b}(\Gamma \backslash G, \Omega^{p}(E_{\alpha}))$.
The sheaf  $R^{b}\pi_{\ast} \Omega^{p}(E_{\alpha})$ over $\Lambda\backslash\R^{s}$ is the sheafification  of the pre-sheaf such that each open set 
$O\subset \Lambda\backslash\R^{s}$  corresponds  to  the vector space $H^{b}(\pi^{-1}(O), \Omega^{p}(E_{\alpha}))$.
Since the flat bundle $E_{\alpha}$ corresponds to a homomorphism $\alpha_{\vert\Lambda}\in Hom(\Lambda,\C^{\ast})$,
as a sheaf on $\Gamma \backslash G$,  $E_{\alpha}$ is constant on $\pi^{-1}(O)$ for sufficiently small open set $O\subset  \Lambda\backslash\R^{s}$.
Thus we have $ R^{b}\pi_{\ast} \Omega^{p}(E_{\alpha})\cong (R^{b}\pi_{\ast} \Omega^{p})\otimes \tilde {E}_{\alpha_{\vert\Lambda}}$ where $\tilde {E}_{\alpha_{\vert\Lambda}}$ is the local system on $\Lambda\backslash\R^{s}$ corresponding to $\alpha_{\vert\Lambda}\in Hom(\Lambda,\C^{\ast})$.
\cite[Lemma 4.3]{OTH} says that $R^{b}\pi_{\ast} \Omega^{p} $ is a local system on the torus $ \Lambda\backslash\R^{s}$ so that locally 
$R^{b}\pi_{\ast} \Omega^{p}$ is isomorphic to $\bigwedge^{p}\C^{s+t}\otimes \bigwedge^{b} \C^{t}$ and $R^{b}\pi_{\ast} \Omega^{p}$ corresponds to a diagonal representation of $\Lambda=p(l(U))\cong U$.
Let $E$ be a local system on $\Lambda\backslash\R^{s}$ corresponding to a $1$-dimensional complex representation of $\Lambda$.
It is well known that $H^{\ast}(\Lambda\backslash\R^{s},E)\cong \bigwedge \C^{s}$ for trivial  $E$ otherwise $H^{\ast}(\Lambda\backslash\R^{s},E)=0$.
By this for any $E$, we have
\[\bigoplus_{\beta\in Hom(\Lambda,\C^{\ast})} H^{\ast}(\Lambda\backslash\R^{s},E\otimes \tilde {E}_\beta)= H^{\ast}(\Lambda\backslash\R^{s},E\otimes E^{-1})\oplus \bigoplus_{\tilde{E}_{\beta}\not=E^{-1}} H^{\ast}(\Lambda\backslash\R^{s},E\otimes \tilde {E}_\beta)\cong \bigwedge \C^{s}.
\]
Thus, identifying ${\mathcal A}_{(G,N)}(\Gamma)$  with $Hom(\Lambda,\C^{\ast})$, we have 
\[\bigoplus_{E_{\alpha}\in{\mathcal A}_{(G,N)}(\Gamma)} E^{a,b}_{2}(\Omega^{p}(E_{\alpha}))\cong \bigoplus_{\beta\in Hom(\Lambda,\C^{\ast})} H^{a}(\Lambda\backslash\R^{s},  (R^{b}\pi_{\ast} \Omega^{p})\otimes \tilde {E}_\beta)\cong \bigwedge^{a} \C^{s} \otimes \bigwedge^{p}\C^{s+t}\otimes \bigwedge^{b} \C^{t}.\]
We have
\[ \bigoplus_{a+b=q}\bigoplus_{E_{\alpha}\in{\mathcal A}_{(G,N)}(\Gamma)} E^{a,b}_{2}(\Omega^{p}(E_{\alpha})) \cong  \bigwedge^{p}\C^{s+t}\otimes \bigwedge^{q} \C^{s+t}.\]
By Proposition \ref{inj}, we have 
\[\sum_{a+b=q}\sum_{E_{\alpha}\in{\mathcal A}_{(G,N)}(\Gamma)} \dim E^{a,b}_{2}(\Omega^{p}(E_{\alpha}))= \dim \bigwedge^{p}\C^{s+t}\otimes \bigwedge^{q} \C^{s+t}\le \sum_{E_{\alpha}\in{\mathcal A}_{(G,N)}(\Gamma)}\dim H^{p,q}(\Gamma \backslash G, E_{\alpha}).\]
Since $E^{a,b}_{r}(\Omega^{p}(E_{\alpha}))$ converges to $H^{a+b}(\Gamma \backslash G, \Omega^{p}(E_{\alpha}))\cong H^{p,a+b}(\Gamma \backslash G, E_{\alpha})$,
 the converse inequality holds by the standard argument on the spectral sequence.
Thus, the Leray spectral sequence $E^{\ast,\ast}_{\ast}(\Omega^{p}(E_{\alpha}))$ degenerates at $E_{2}$-term and so
we have
\[ \dim \bigwedge^{p}\C^{s+t}\otimes \bigwedge^{q} \C^{s+t}= \sum_{E_{\alpha}\in{\mathcal A}_{(G,N)}(\Gamma)}\dim H^{p,q}(\Gamma \backslash G, E_{\alpha}).\]
Hence the injection in  Proposition \ref{inj} is an isomorphism.
\end{proof}

Since $U\cong p(l(U))=\Lambda$ is embedded in $\R^{s}$ as a lattice,
the set ${\mathcal A}_{(G,N)}(\Gamma)$ is equal to  the set ${\mathcal A}(U)$ of isomorphism classes of flat complex line bundles over  $X(K,U)=\Gamma \backslash G$ given by homomorphisms in ${\rm Hom}(U,\C^{\ast})$.
We have the following consequence of Proposition \ref{inj}.
 \begin{cor}\label{Pcor}
 For any integer $r$, we have
\[\dim \bigwedge^{r}\C^{2s+2t} = \sum_{E\in {\mathcal A}(U)}\dim H^{r}(X(K,U),E).
\] 
\end{cor}
\begin{rem}
See \cite{IO} for the de Rham cohomology of OT-manifolds with values in trivial and some specific flat line bundles.
\end{rem}

Theorem \ref{Dol} gives the following statement.
\begin{cor}
For any integers $p,q$,  we  have 
   \[\dim \bigwedge ^{p}\C^{s+t}\otimes \bigwedge^{q}\C^{s+t} = \sum_{E\in {\mathcal A}(U)}\dim H^{p,q}(X(K,U),E).\] 
    \end{cor}

\section{Cohomology of holomorphic  line bundles over OT-manifolds:vanishing and non-vanishing}

For each $E_{\alpha}\in {\mathcal A}_{(G,N)}(\Gamma)$, by Theorem \ref{Dol},
we have 
\[H^{p,q}(\Gamma \backslash G, E_{\alpha})\cong \left\langle \alpha_{I}\wedge \bar\alpha_{J}\wedge \beta_{K}\wedge \bar{\beta}_{L} \otimes v_{e^{\Psi_{IK\bar{L}}(x)}}
\left\vert \begin{array}{cc}\vert I\vert+ \vert K\vert =p \text{ and } \vert J \vert + \vert L\vert = q  \\
E_{\alpha}=E_{e^{\Psi_{IK\bar{L}}(x)}}
 \end{array}\right. 
  \right\rangle.
\]

\begin{cor}\label{Van}
$H^{p,q}(\Gamma \backslash G, E_{\alpha})\not=0$ if and only if  for some multi-indices  $I\subset [s], K,L\subset [t]$ with $\vert I\vert+\vert K\vert=p$ and $\vert L \vert\le q$, $E_{\alpha}=E_{e^{\Psi_{IK\bar{L}}(x)}}$.
If $E_{\alpha}=E_{e^{\Psi_{IK\bar{L}}(x)}}$, then we have
\[\dim H^{p,q}(\Gamma \backslash G, E_{\alpha})\ge \left( \begin{array}{cc}
s  \\
q-\vert L \vert
\end{array}\right)
\]
where $ \left( \begin{array}{cc}
n  \\
k
\end{array}\right)$ means the number of $k$-combinations.

\end{cor}
We notice that  $E_{\alpha}=E_{e^{\Psi_{IK\bar{L}}(x)}}$ if and only if $\alpha(x)=e^{\Psi_{IK\bar{L}(x)}}$ for any $x\in \Lambda$.
For the trivial $E_{\alpha}$, this Corollary gives \cite[Corollary 3.5]{KRO}.

 For $u\in U$ with $x=p(l(u))\in p(l(U))$ we have $\sigma_{i}(u)=e^{x_{i}}$ for $1\le i\le s$,  $\sigma_{s+k}(u)=e^{\psi_{k}(x)}$ and $\sigma_{s+t+k}(u)=e^{\bar\psi_{k}(x)}$ for  $1\le k\le s$.
Hence we have  
\[e^{\Psi_{IK\bar{L}}(x)}=\prod_{i\in I}\sigma_{i}(u)\prod_{j\in K}\sigma_{s+k}(u)\prod_{l\in L}\sigma_{s+t+l}(u).\]

\begin{cor}
Let $E$ be a flat complex line bundle over an OT-manifold $X(K,U)$ corresponding to $\rho\in {\rm Hom}(U,\C^{\ast})$.
Then  \[
H^{p,q}(X(K,U),E)\not=0\]
 if and only if  for some multi-indices  $I\subset [s], K,L\subset [t]$   with $\vert I\vert+\vert K\vert=p$ and $\vert L \vert\le q$,  we have
 \[\rho(u)=\prod_{i\in I}\sigma_{i}(u)\prod_{k\in K}\sigma_{s+k}(u)\prod_{l\in L}\sigma_{s+t+l}(u)\]
  for any $u\in U$.
  If 
  \[\rho(u)=\prod_{i\in I}\sigma_{i}(u)\prod_{k\in K}\sigma_{s+k}(u)\prod_{l\in L}\sigma_{s+t+l}(u)\]
   for any $u\in U$,
 then  we have 
   \[\dim H^{p,q}(X(K,U),E)\ge \left( \begin{array}{cc}
s  \\
q-\vert L \vert
\end{array}\right).
\]

\end{cor}
\begin{rem}\label{APV}
This statement implies \cite[Theorem 3.1]{APV}.
Actually, for $p=0$ and $q=1$,  $
H^{0,1}(X(K,U),E)\not=0$
 if and only if $\rho$ is trivial or $\rho=\sigma_{s+t+1},\dots , \sigma_{s+2t}$.
 This seems different from  \cite[Theorem 3.1]{APV}.
But we may remark
 that we use the left action  but on the other hand
  in  \cite{APV} the right action is used. 
The correspondence between the right-quotient and left-quotient is given by the inverse.
\end{rem}

\begin{example}
We consider the case $t = 1$.
In this case, any $U$ is a finite-index subgroup of ${\mathcal O}_{K}^{\ast\, +}$.
For our solvmanifold presentation $\Gamma \backslash G$ of an OT-manifold $X(K,U)$,
we can write
\[\psi_{1}(x)=-\frac{1}{2}(x_{1}+\dots+x_{s})+ \sqrt{-1}\varphi(x)
\]
for some real linear function $\varphi(x)$.
Thus, for  for  multi-indices  $I\subset [s], K,L\subset [t=1]$,
we have 
\[\Psi_{IK\bar{L}}(x)=\left\{ \begin{array}{cccc}
\sum_{i\in I} x_{i}\qquad (K=L=\o)\\
\frac{1}{2}\sum_{i\in I} x_{i}-\frac{1}{2}\sum_{i\in \check{I}} x_{i}+ \sqrt{-1}\varphi(x) \qquad (K=\{1\},L=\o )\\
\frac{1}{2}\sum_{i\in I} x_{i}-\frac{1}{2}\sum_{i\in \check{I}} x_{i}- \sqrt{-1}\varphi(x) \qquad (K=\o,L=\{1\} )\\
-\sum_{i\in \check{I}} x_{i}  \qquad (K=\{1\},L=\{1\} )
 \end{array}\right. .
\]
 We can say that 
 $E_{e^{\Psi_{IK\bar{L}}(x)}}=E_{e^{\Psi_{I^{\prime}K^{\prime}\bar{L^{\prime}}}(x)}}$ if and only if $(I,K,L)=(I^{\prime},K^{\prime},L^{\prime})$,  $(I,K,L)=(\o,\o,\o)$ and  $(I^{\prime},K^{\prime},L^{\prime})=([s],\{1\},\{1\})$ or $(I,K,L)=([s],\{1\},\{1\})$ and  $(I^{\prime},K^{\prime},L^{\prime})=(\o,\o,\o)$.

We compute $H^{p,q}(\Gamma \backslash G, E_{e^{\Psi_{IK\bar{L}}(x)}})$ for each $I,K,L$ so that $E_{e^{\Psi_{IK\bar{L}}(x)}}$ is non-trivial i.e. $(I,K,L)\not=(\o,\o,\o), ([s],\{1\},\{1\})$.
If $K=L=\o$,
\[H^{p,q}(\Gamma \backslash G, E_{e^{\Psi_{IK\bar{L}}(x)}})=\left\{ \begin{array}{cccc}
\langle \alpha_{I}\rangle \wedge \bigwedge^{q}\langle \bar{\alpha}_{1},\dots,  \bar{\alpha}_{s} \rangle \qquad (p=\vert I\vert)\\
0\qquad ({\rm otherwise})
 \end{array}\right. .
\]
If $K=\{1\},L=\o$, 
\[H^{p,q}(\Gamma \backslash G, E_{e^{\Psi_{IK\bar{L}}(x)}})=\left\{ \begin{array}{cccc}
\langle \alpha_{I}\wedge\beta_{1}\rangle \wedge \bigwedge^{q}\langle \bar{\alpha}_{1},\dots,  \bar{\alpha}_{s} \rangle \qquad (p=\vert I\vert+1)\\
0\qquad ({\rm otherwise})
 \end{array}\right. .
\]
If $K=\o,L=\{1\} $, 
\[H^{p,q}(\Gamma \backslash G, E_{e^{\Psi_{IK\bar{L}}(x)}})=\left\{ \begin{array}{cccc}
\langle \alpha_{I}\rangle \wedge \bigwedge^{q-1}\langle \bar{\alpha}_{1},\dots,  \bar{\alpha}_{s} \rangle\wedge \langle\bar{\beta}_{1}\rangle \qquad (p=\vert I\vert)\\
0\qquad ({\rm otherwise})
 \end{array}\right. .
\]
If $K=\{1\},L=\{1\}$,
\[H^{p,q}(\Gamma \backslash G, E_{e^{\Psi_{IK\bar{L}}(x)}})=\left\{ \begin{array}{cccc}
\langle \alpha_{I}\wedge\beta_{1}\rangle \wedge \bigwedge^{q-1}\langle \bar{\alpha}_{1},\dots,  \bar{\alpha}_{s} \rangle  \wedge \langle\bar{\beta}_{1}\rangle\qquad (p=\vert I\vert+1)\\
0\qquad ({\rm otherwise})
 \end{array}\right. .
\]
In particular for any $I,K,L$ with $(I,K,L)\not=(\o,\o,\o), ([s],\{1\},\{1\})$, $H^{0,q}(\Gamma \backslash G, E_{e^{\Psi_{IK\bar{L}}(x)}})=0$.
\begin{rem}
In this case, the equality 
\[\dim H^{p,q}(\Gamma \backslash G, E_{e^{\Psi_{IK\bar{L}}(x)}})= \left( \begin{array}{cc}
s  \\
q-\vert L \vert
\end{array}\right)
\]
holds (see  the inequality as in Corollary \ref{Van}).
\end{rem}

As noted in Remark \ref{tan}, we have
\[\Theta \cong E_{e^{x_{1}}}\oplus \dots \oplus E_{e^{x_{s}}}\oplus E_{e^{\psi_{1}(x)}}\oplus \dots \oplus E_{e^{\psi_{t}(x)}}.\]
Hence, for $0<p\le s+1$,  we have 
\[H^{0,q}(M,\bigwedge ^{p}\Theta)\cong \bigoplus_{\vert I\vert +\vert K\vert=p, L=\o} H^{0,q}(\Gamma \backslash G, E_{e^{\Psi_{IK\bar{L}}(x)}})=0.
\]
This means the following statement.
\begin{prop}
For an OT-manifold $X(K,U)$ with $t=1$,
we have 
\[H^{0,q}(X(K,U),\bigwedge ^{p}\Theta)=0.
\]
\end{prop}
Since $H^{0,1}(X(K,U),\Theta)=0$, every OT-manifold $X(K,U)$ with $t=1$ is rigid (cf. \cite{APV}).
Moreover by $H^{0,0}(X(K,U),\bigwedge ^{2}\Theta)=0$,  $X(K,U)$ with $t=1$ does not admit non-zero holomorphic Poisson structure.

Generalized complex structures are geometric structures including complex structures and symplectic structures as special cases introduced by Hitchin and developed by Gualtieri \cite{Gua}.
It is known that small  deformations of a  complex structure regarded as a generalized complex structure can be controlled   by the following three parameters  (see \cite[Section 5]{Gua}):
\begin{itemize}
\item Holomorphic Poisson structures;
\item Usual deformations of the  complex structure;
\item $B$-field transformations.
\end{itemize}
Thus,   for  every OT-manifold $X(K,U)$ with $t=1$,   every  small  deformation of the  complex structure on $X(K,U)$
 regarded as a generalized complex structure is given by  a  $B$-field transformation.
\end{example}

\begin{example}[\cite{O3}]
In \cite[Section 3.1]{O3}, Otiman gives a field  $K$ and a subgroup $U\subset {\mathcal O}_{K}^{\ast +}$ such that 
$s=t=2$ and  for any $u\in U$, $\sigma_{1}(u)\sigma_{3}(u)\sigma_5(u)=\sigma_{2}(u)\sigma_{4}(u)\sigma_6(u)=1$.
By these relations, we can write
\[\psi_{1}(x)=-\frac{1}{2}x_{1}+ \sqrt{-1}\varphi_{1}(x), \psi_{2}(x)=-\frac{1}{2}x_{2}+ \sqrt{-1}\varphi_{2}(x)
\]
for some real linear functions $\varphi_{1}(x), \varphi_{2}(x)$.
Thus, for examples,  we compute
\[H^{p,q}(\Gamma \backslash G )=\left\{ \begin{array}{cccc}
 \bigwedge^{q}\langle \bar{\alpha}_{1},  \bar{\alpha}_{2} \rangle \qquad (p=0)\\
\langle \alpha_{1}\wedge \beta_{1}\wedge\bar\beta_{1},\alpha_{2}\wedge \beta_{2}\wedge\bar\beta_{2}  \rangle  \wedge\bigwedge^{q-1}\langle \bar{\alpha}_{1},  \bar{\alpha}_{2} \rangle    \qquad (p=2,q\ge 1)\\
\langle \alpha_{1}\wedge\alpha_{2}\wedge \beta_{1}\wedge \beta_{2}\wedge\bar\beta_{1}\wedge\bar\beta_{2}  \rangle  \wedge\bigwedge^{q-2}\langle \bar{\alpha}_{1},  \bar{\alpha}_{2} \rangle    \qquad (p=4,q\ge 2)\\
0\qquad ({\rm otherwise})
 \end{array}\right. 
\]
and 
\[H^{p,q}(\Gamma \backslash G ,E_{e^{x_{1}}})=\left\{ \begin{array}{cccc}
 \langle \alpha_{1}\rangle\wedge\bigwedge^{q}\langle \bar{\alpha}_{1},  \bar{\alpha}_{2} \rangle \qquad (p=1)\\
\langle \alpha_{1}\wedge\alpha_{2}\wedge \beta_{2}\wedge\bar\beta_{2}  \rangle  \bigwedge^{q-1}\langle \bar{\alpha}_{1},  \bar{\alpha}_{2} \rangle    \qquad (p=3,q\ge 1)\\
0\qquad ({\rm otherwise})
 \end{array}\right. .
\]
Hence, for $I=K=J=\{1\}$ and $q\ge 2$, we have 
\[H^{2,q}(\Gamma \backslash G, E_{e^{\Psi_{IK\bar{L}}(x)}})=H^{2,q}(\Gamma \backslash G )=2\left( \begin{array}{cc}
2  \\
q-1
\end{array}\right)>\left( \begin{array}{cc}
2  \\
q-1
\end{array}\right).
\]
\end{example}

\end{document}